\documentclass[reqno, 10pt]{amsart} 

\usepackage{latexsym,amssymb,amsmath,bm,graphicx,enumitem,verbatim}
\usepackage{tikz}
\usepackage[disable, textsize=footnotesize,color=white!40, bordercolor=white]{todonotes}

\usepackage{float}

\usepackage{stmaryrd, bbm} 

\usepackage[urlcolor=blue]{hyperref}

\usepackage
[a4paper, margin=3.048 cm, marginparwidth=70 pt]{geometry}

\usetikzlibrary{positioning,decorations.pathmorphing,fit,petri,backgrounds,decorations.pathreplacing}

\makeatletter
\@namedef{subjclassname@2020}{\textup{2020} Mathematics Subject Classification}
\makeatother

\DeclareMathOperator{\Lim}{Lim}

\DeclareMathOperator{\Ord}{Ord}

\DeclareMathOperator{\dom}{dom}

\DeclareMathOperator{\cof}{cof}
\DeclareMathOperator{\1}{\mathbbm{1}}

\DeclareMathOperator{\CH}{CH}
\DeclareMathOperator{\ZFC}{ZFC}

\newcommand{\Col}{\mathrm{Col}}

\DeclareMathOperator{\ran}{\mathrm{ran}}

\DeclareMathOperator{\PFA}{\mathsf{PFA}}

\newcommand{\PP}{\mathbb P}
\newcommand{\QQ}{\mathbb Q}

\newcommand{\bi}{\begin{itemize}} 
\newcommand{\ei}{\end{itemize}} 

\newcommand{\Add}{\mathrm{Add}} 

\definecolor{dblue}{rgb}{0,0,0.70}
\hypersetup{
	unicode=true,
	colorlinks=true,
	citecolor=dblue,
	linkcolor=dblue,
	anchorcolor=dblue
}

\newenvironment{enumerate-(a)}{\begin{enumerate}[label={\upshape (\alph*)}]}{\end{enumerate}}

\newenvironment{enumerate-(a)-r}{\begin{enumerate}[label={\upshape (\alph*)},resume]}{\end{enumerate}}

\newenvironment{enumerate-(A)}{\begin{enumerate}[label={\upshape (\Alph*)}]}{\end{enumerate}}

\newenvironment{enumerate-(A)-r}{\begin{enumerate}[label={\upshape (\Alph*)},resume]}{\end{enumerate}}

\newenvironment{enumerate-(i)}{\begin{enumerate}[label={\upshape (\roman*)}]}{\end{enumerate}}

\newenvironment{enumerate-(i)-r}{\begin{enumerate}[label={\upshape (\roman*)},resume]}{\end{enumerate}}

\newenvironment{enumerate-(I)}{\begin{enumerate}[label={\upshape (\Roman*)}]}{\end{enumerate}}

\newenvironment{enumerate-(I)-r}{\begin{enumerate}[label={\upshape (\Roman*)},resume]}{\end{enumerate}}

\newenvironment{enumerate-(1)}{\begin{enumerate}[label={\upshape (\arabic*)}]}{\end{enumerate}}

\newenvironment{enumerate-(1)-r}{\begin{enumerate}[label={\upshape (\arabic*)},resume]}{\end{enumerate}}

\renewcommand{\bar}{\overline}

\newcommand{\tk}{{}^\kappa2} 
\newcommand{\kk}{{}^\kappa\kappa} 
\newcommand{\len}{\mathrm{l}}

\newtheorem{fact}{Fact}[section]
\newtheorem{lemma}[fact]{Lemma}
\newtheorem{theorem}[fact]{Theorem}

\newtheorem{claim}[fact]{Claim}
\newtheorem*{claim*}{Claim}

\newtheorem*{subclaim*}{Subclaim}
\newtheorem{conjecture}[fact]{Conjecture}

\newtheorem{proposition}[fact]{Proposition}

\theoremstyle{definition}

\newtheorem*{problem*}{Problem}
\newtheorem{question}[fact]{Question}
\newtheorem{definition}[fact]{Definition}

\theoremstyle{remark}
\newtheorem{remark}[fact]{Remark}

\title[]{Failure of the $G_0$-dichotomy for generalized Cantor spaces}

\date{\today} 

\author{Matteo Casarosa}
\address{Departament de Matem\`atiques i Inform\`atica, Universitat de Barcelona, Gran Via
de les Corts Catalanes 585, 08007 Barcelona, Catalonia} 

\email{matteo.casarosa@ub.edu}

\author{Philipp Schlicht}
\address{Dipartimento di ingegneria dell'informazione e scienze matematiche, Università di Siena, via Roma 56, 53100 Siena, Italy} 
\email{philipp.schlicht@unisi.it}

\thanks{
The first-listed author was partially supported by the Universit\'e Franco-Italienne's mobility funding grant Vinci Chapter 2, 2024, and by the Department of Mathematics and Computer Science of the Universitat de Barcelona. 
The second-listed author gratefully acknowledges the support of Istituto Nazionale di Alta Matematica, Gruppo Nazionale per le Strutture Algebriche, Geometriche e le loro Applicazioni (INdAM-GNSAGA). 
For the purpose of open access, the authors have applied a ‘Creative Commons Attribution' (CC BY) public copyright licence to any Author Accepted Manuscript (AAM) version arising from this submission.
No AI tools were using in this work. 
}

\begin{document}

\begin{abstract} 
Kechris, Solecki and Todor\v{c}evi\'c's $G_0$-dichotomy characterizes Borel graphs that admit a Borel measurable coloring with countably many colors. 
We show that the analogue to the $G_0$-dichotomy for generalized Cantor spaces fails at the lowest possible complexity, namely for closed graphs. 
%
\end{abstract}

\maketitle

\setcounter{tocdepth}{1}


\section{Introduction}

One may ask under which circumstances a graph on the Cantor space ${}^\omega2$ of sequences with values $0$ and $1$ indexed by natural numbers admits a Borel measurable coloring with countably many colors. 
Kechris, Solecki and Todor\v{c}evi\'c showed that among Borel (and even analytic) graphs $G$ that do not admit a Borel measurable coloring with countably many colors, there is a least one $G_0$ \cite{kechris1999borel}. 
More precisely, the $G_0$-dichotomy states that any analytic graph $G$ on a Polish space either admits a Borel measurable coloring with at most countably many colours, or otherwise there exists a continuous graph homomorphism from $G_0$ to $G$. 
It is an important result about definable graphs and a widely used tool in descriptive set theory with various applications including Silver's dichotomy for coanalytic equivalence relations and the Lusin-Novikov uniformization theorem for Borel relations with countable sections \cite{millerparisone}. 
The study of colorings of definable graphs has led to the field of  decriptive combinatorics \cite{kechris2016descriptive}. 

Here we study the more general setting where the cardinality $\omega$ of the natural numbers is replaced by an arbitrary uncountable regular cardinal $\kappa$ with $\kappa^{<\kappa}=\kappa$. 
Generalized Cantor spaces $\tk$ and generalized Baire spaces $\kk$ have been studied intensively because of their close connections with model theory, classification theory \cite{MR3235820,moreno2023shelah} and combinatorial set theory \cite{lucke2020descriptive}. 
Since the graph-theoretic method is powerful in descriptive set theory, it is desirable to extend it to generalized descriptive set theory. 
This was initiated in \cite{schlicht2023open}
where an analogue to the open graph dichotomy was studied. 
Here we study the analogue of $G_0$ for ${}^\kappa2$, denoted $G_0^\kappa$. 
Soon after Miller's recasting of the $G_0$-dichotomy and variants thereof \cite{millerparisone}, the problem emerged whether these techniques extend to the uncountable setting. 
Friedman and T\"ornquist, among others, asked in 2010 whether the analogue to the $G_0$-dichotomy for $\kappa$-Borel graphs on ${}^\kappa 2$ is consistent. 
Our main result provides a negative answer. 
Precise definitions will follow. 


\begin{theorem}(see Theorem \ref{G0 dichotomy fails for Gwo}) 
\label{theorem G0 dichotomy fails}
The $G_0^\kappa$-dichotomy fails for closed graphs on $\tk$. 
\end{theorem} 

This is proved in Section \ref{section a closed counterexample}. 
Closed graphs are optimal, since the $G_0^\kappa$-dichotomy for open graphs follows from the open graph dichotomy for graphs on ${}^\kappa2$, and the latter is consistent relative to an inaccessible cardinal by \cite{schlicht2023open}. 

One may look for simpler counterexamples than the one in Theorem \ref{theorem G0 dichotomy fails}, in particular for induced subgraphs of $G_0^\kappa$. 
The next result shows that in the constructible universe $L$, induced subgraphs on closed subsets can violate the $G_0^\kappa$-dichotomy.

\begin{theorem}(see Theorem \ref{counterexample in L}) 
Suppose that $V=L$ and $\kappa$ is an uncountable regular cardinal. 
Then there exists a slim $\kappa$-Kurepa tree $T$ such that 
\begin{enumerate-(a)} 
\item 
$G_0^\kappa{\upharpoonright} [T]$ is a counterexample to the $G_0^\kappa$-dichotomy. 

\item 
$G_0^\kappa{\upharpoonright} [T]$ is not minimal among $\kappa$-Borel graphs that do not admit $\kappa$-Borel colorings with $\kappa$ many colors. 

\end{enumerate-(a)}
\end{theorem} 

This is proved in Section \ref{section induced subgraph}. 
An alternative way of obtaining counterexamples to the $G_0^\kappa$-dichotomy via counterexamples to Silver's dichotomy is discussed in Section \ref{section a connection with Silver's dichotomy}. 
For the Cantor space, the $G_0$-dichotomy for a (sufficiently closed) class of sets implies Silver's dichotomy for the complementary class \cite{millerparisone}. 
We produce a similar argument for generalized Cantor spaces 
in order to generate further $\kappa$-Borel counterexamples to the $G_0^\kappa$-dichotomy in models such as $L$ where Silver's dichotomy for $\kappa$-Borel equivalence relations fails. 

These results delimit dichotomies for closed graphs and suggest to first study more specific classes of graphs such as locally small graphs and induced subgraphs of $G_0^\kappa$. 
For the latter, one needs to work with models without $\kappa$-Kurepa trees as in Section \ref{section induced subgraph} to avoid   counterexamples. 
\section{Preliminaries}


Suppose that $\kappa$ is an uncountable cardinal with $\kappa^{<\kappa}=\kappa$. 
Let $\mathrm{l}(t)$ denote the length of a sequence $t\in 2^{<\kappa}$. 
A graph on a set $X$ is a subset of $(X\times X)\setminus \Delta_X$, where $\Delta_X=\{(x,x)\mid x \in X\}$ is the diagonal. 
A graph on $X$ is called closed if it is a closed subset of $(X\times X)\setminus \Delta_X$. 
Note that for higher complexities, removing the diagonal from the product space is not relevant for calculating the complexity. 
Fix an upwards dense subset $S$ of $2^{<\kappa}$ such that there exist at most one $t\in S$ of length $i$ for each $i<\kappa$. 
Let 
$$G^\kappa_0= G_0^\kappa(S)= \{(t^\smallfrown 0{}^\smallfrown x,t^\smallfrown 1{}^\smallfrown x) \mid t\in S,\ x\in 2^\kappa \}.$$ 
It is easy see that $G_0^\kappa$ is closed. 
Moreover, suppose that $(x,y) \notin G_0^\kappa$.

$G_0^\kappa$ does not admit a $\kappa$-Borel $\kappa$-colouring by the next lemma. 
Let $\mathrm{Lev}_\alpha(T)$ denote the $\alpha$th level of a tree $T$. 
We call a subtree $T$ of ${}^{<\kappa}2$ \emph{uniformly splitting} if for all $\alpha<\kappa$, $s,t\in \mathrm{Lev}_\alpha(T)$ and $i=0,1$, we have $s^\smallfrown i\in T \Leftrightarrow t^\smallfrown i\in T$. 
A topological space $X$ is called \emph{$\lambda$-Baire} if for any sequence $(D_i \mid i<\kappa)$ of open dense subsets of $X$, the intersection $\bigcap_{i<\kappa} D_i$ is dense.

\begin{lemma}\label{comeager not discrete} 
(see \cite[Proposition 3]{millerparisone})  
Suppose that $T$ is a uniformly splitting subtree of ${}^{<\kappa}2$ such that $S$ is upward dense in $T$ and $[T]$ is $\kappa$-Baire. 
Then any non-$\kappa$-meager subset of $[T]$ with the $\kappa$-Baire property is not $G_0^\kappa$-independent. 
\end{lemma} 
\begin{proof} 
Since $A$ has the $\kappa$-Baire property and is not $\kappa$-meager, there is some $t\in T$ such that $A$ is $\kappa$-comeager in $N_t\cap [T]$. 
Since $S$ is upwards dense in $T$, there is some $u\supseteq t$ in $S\cap T$ and then $A$ is comeager in $N_u\cap [T]$. 
Let $i=\len(u)$. 
Since $T$ is uniformly splitting, the map $f_i\colon \tk\rightarrow \tk$ which flips only the $i$-th digit is an auto-homeomorphism of $[T]$. 
It follows that there exists some $x\in 2^\kappa$ such that $u^\smallfrown 0{}^\smallfrown x, u^\smallfrown 1{}^\smallfrown x \in A$, and therefore $A$ is not $G_0^\kappa$-independent. 
\end{proof}

The next lemma shows that versions of $G_0^\kappa$ relative to different sets $S$ are all equivalent for our purposes. 

\begin{lemma} 
Suppose that $S_0$ and $S_1$ are dense subsets of ${}^{<\kappa}2$ with at most one sequence of each length.  
Then there exists an 
\todo[color=yellow!30]{Can we make this a reduction?} a
injective continuous homomorphism from $G_0^\kappa(S_0)$ to $G_0^\kappa(S_1)$ that is induced by the canonical isomorphism between ${}^{<\kappa}2$ and a uniformly splitting subtree of ${}^{<\kappa}2$. 
\end{lemma}
\begin{proof} 
We construct an order and incompatibility preserving family $(t_s\mid s\in {}^{<\kappa}2)$. 
Let $t_\emptyset=\emptyset$. 
In limit steps, we form unions. 
In successor steps, suppose that $t_s$ has been constructed. 
If there is no $s\in S_0$ of length $\alpha$, let $t_{s^\smallfrown i} = t_s^\smallfrown i$ for all $s\in {}^\alpha2$ and  $i=0,1$. 
Now suppose that there is a (necessarily unique) $s\in S_0$ of length $\alpha$. 
Using density of $S_1$, find some $u\supseteq t_{s_\alpha}$ in $S_1$ and write $u=t_{s_\alpha}^\smallfrown v$. 
Let $t_{r^\smallfrown i}= t_r ^\smallfrown v^\smallfrown i$ for all $r\in {}^\alpha 2$ and $i=0,1$. 
Finally, define $f\colon \tk\rightarrow \tk$ by letting $f(x)=\bigcup_{\alpha<\kappa} t_{x{\upharpoonright}\alpha}$. Then $f$ is as required.  
\end{proof}

\section{A closed counterexample} 
\label{section a closed counterexample} 

The argument for $G_0^\kappa$ below is motivated by the argument for the following result.  
Let $E_0^\kappa$ denote the equivalence relation consisting of all pairs $(x,y)\in 2^\kappa\times 2^\kappa$ such that $x(i)=y(i)$ for all except finitely many $i<\kappa$. 
Then every connected component of $G_0^\kappa$ is contained in an $E_0^\kappa$-equivalence class.

\begin{proposition}\label{no selector for E0kappa} 
\todo[color=yellow!30]{Does $E_0^\kappa$ admit a selector with a $\kappa$-Borel graph?} 
$E_0^\kappa$ is smooth, but it does not admit a $\kappa$-Borel measurable selector. 
Moreover, in any $\Add(\kappa,1)$-generic extension, $E_0^\kappa$ admits no selector definable from a ground model element. 
\end{proposition} 
\begin{proof} 
To see that $E_0^\kappa$ is smooth, let $f\colon \tk\rightarrow \tk$ be the function that sends $x$ to the set $f(x)$ of all initial segments of $x$ and their finite changes. 
Since $f(x)$ is a subset of $2^{<\kappa}$ and we assume $2^{<\kappa}=\kappa$, we can identify the range of $f$ with $2^\kappa$. Clearly $f$ is continuous. 

We next show that in any $\Add(\kappa,1)$-generic extension, $E_0^\kappa$ admits no selector definable from a ground model element. 
We will then prove the remaining claim using this.  
Suppose towards a contradiction that $x$ is $\Add(\kappa,1)$-generic over $V$ and $f$ is a selector definable from a ground model element. 
Suppose that $p\subseteq x$ forces that $f(x) \triangle x = t$ for some finite subset $t$ of $\kappa$. 
Since $x$ is $\Add(\kappa,1)$-generic, it has unboundedly many initial segments in $S$ by a density argument. 
Hence $p\subseteq x{\upharpoonright}\alpha \in S$ for some $\alpha<\kappa$. 
Let $y\in \tk$ be the function with $x(i)=y(i)$ for all $i\neq \alpha$, but $x(\alpha)\neq y(\alpha)$. 
Then $V[x]=V[y]$. 
Moreover, $y$ is $\Add(\kappa,1)$-generic over $V$ and therefore $f(y) \triangle y = t$. 
Using that $f$ is a selector and $(x,y)\in E_0^\kappa$, we thus have $f(x)=f(y)$ and $x =f(x)\triangle t=f(y)\triangle t=y$, contradicting the fact that $x$ and $y$ differ at $\alpha$. 

Finally, we show that $E_0^\kappa$ does not admit a $\kappa$-Borel measurable selector. 
Suppose towards a contradiction that $f$ is a $\kappa$-Borel measurable selector for $E_0^\kappa$. 
We can code $f$ by listing each $t\in 2^{<\kappa}$ together with a $\kappa$-Borel code for $f^{-1}(N_t)$. 
Using a binary relation on $\kappa$ coding $(H_\kappa,\in)$, this list can be coded by a subset $x$ of $\kappa$. 
We claim that the statement that $x$ codes a function that is a selector for $E_0^\kappa$ is $\Pi^1_1(\tk)$ in $x$. 
This is easy to see for the statement that $x$ codes a function. 
Moreover, the statement that $x$ codes a selector for $E_0^\kappa$ says that for all $x, y \in 2^\kappa$ $f(x)\triangle x$ is finite, and 
$x\triangle y$ is finite if and only if $f(x)=f(y)$.  
Finiteness of the symmetric difference is arithmetic over $(H_\kappa,\in)$, so the claim holds. 
By $\Pi^1_1(\tk)$-absoluteness for ${<}\kappa$-closed generic extensions \cite[Proposition 7.3]{lucke2012definability}, the function coded by $x$ in $V[G]$ is a $\kappa$-Borel measurable selector for $E_0^\kappa$. 
But this contradicts the previous claim. 
\end{proof} 

It is easy to see that $E_0^\kappa$ is induced by a continuous action of a group of size $\kappa$ on $\tk$. 
This contrasts with the situation for Polish spaces, since any smooth orbit equivalence relation of a Polish group action on a Polish space admits a Borel measurable selector by a result of Kechris \cite[Theorem 5.4.11]{gao2008invariant}.

The following closed graph was defined by Szir\'aki in \cite[Example 5.3.5]{schlicht2023open}. 
We will show that this graph is a counterexample to the $G_0^\kappa$-dichotomy. 
Let $\mathbb{L}$ be a ${<}\kappa$-dense linear order on $\kappa$ in the sense that there are no $({<}\kappa,{<}\kappa)$-gaps. 
The vertex set for our graph is the set $V_{\mathrm{wo}}$ of subsets of $\kappa$ which are well-ordered in $\mathbb{L}$, and $G_{\mathrm{wo}}$ is the graph on $V_{\mathrm{wo}}$ given by strict end extension. $G_{\mathrm{wo}}$ is closed, since for distinct $x$ and $y$ in the domain of $G_{\mathrm{wo}}$, $x$ is not a proper initial segment of $y$ if and only if there exists an element of $x$ and an element of $y\setminus x$ less or equal to it. 

\begin{proposition}\label{prop graph of wellorders}\ 
\begin{enumerate-(1)} 
\item\label{prop graph of wellorders 1} 
$G_{\mathrm{wo}}$ does not admit a $\kappa$-coloring.

\item\label{prop graph of wellorders 2} 
There exists no continuous homomorphism from $G_{\mathrm{wo}}$ to $G_0^\kappa$.

\end{enumerate-(1)} 
\end{proposition} 
\begin{proof} 
\ref{prop graph of wellorders 1} 
Suppose otherwise. 
Then $V_{\mathrm{wo}}=\bigcup_{\alpha<\kappa} X_\alpha$ where $X_\alpha$ is a $G_{\mathrm{wo}}$-independent set for each $\alpha<\kappa$. 
We construct a strictly increasing chain
$( x_\alpha\mid \alpha<\kappa)$ in $X$ with respect to strict end extension and a strictly 
$\leq$-decreasing chain 
$( b_\alpha \mid \alpha<\kappa)$ in $L$ 
such that for all $\alpha<\kappa$: 
\begin{enumerate-(i)} 
\item \label{prop graph of wellorders condition} 
$\sup(x_\alpha)<b_\alpha$. 

\item 
$x_\alpha\in X_\alpha$ if there exists $x_\alpha$ and $b_\alpha$ extending $( x_\alpha\mid \alpha<\kappa)$ and $( b_\alpha\mid \alpha<\kappa)$ such that \ref{prop graph of wellorders condition} holds. 
\end{enumerate-(i)} 
In step $\beta$, we first choose $b_\beta$ such that $b_\beta<b_\alpha$ for all $\alpha<\beta$ and $\bar{x}_\beta=\bigcup_{\alpha<\beta}x_\alpha$ has an upper bound strictly below $b_\beta$. 
This work since $\mathbb{L}$ is ${<}\kappa$-saturated. 
We then choose a strict end extension $x_\beta$ of $\bar{x}_\beta$ that has an upper bound strictly below $b_\beta$. 
We choose $x_\beta \in X_\beta$ if there exists a pair $(x_\beta,b_\beta)$ with these properties. 
This completes the construction. 
Now let  $x=\bigcup_{\alpha<\kappa}x_\alpha$. 
Since $x$ is a wellorder, $x\in X_\beta$ for some $\beta<\kappa$. 
Therefore we must have chosen $x_\beta\in X_\beta$ in step $\beta$. 
But $x, x_\beta\in X_\beta$ and $x$ is a strict end extension of $x_\beta$, contradicting the fact that $X_\beta$ is $H$-independent. 

\ref{prop graph of wellorders 2} 
Suppose towards a contradiction that $f\colon \tk\rightarrow \tk$ is a continuous homomorphism from $G$ to $G_0^\kappa$. 
Note that $G$ has only one connected component. 
Therefore, the range of $f$ is contained in a single connected component of $G_0^\kappa$ and thus in an equivalence class $C$ of $E_0^\kappa$. 
Since $f^{-1}(y)$ is a $G_{\mathrm{wo}}$-independent set for each $y\in C$ and $|C|=\kappa$, $f$ is a continuous $\kappa$-coloring of $G$. 
This contradicts \ref{prop graph of wellorders 1}. 
\end{proof}

The next lemma describes circumstances in which there exists no continuous homomorphism from an induced subgraph of $G_0^\kappa$ to $G_{\mathrm{wo}}$. 
For a tree $T$ and any $t\in T$, let $T_t=\{u\in T \mid u\subseteq t \vee  t\subseteq u\}$. 
For any $\alpha<\kappa$, let $f_\alpha\colon {}^{\leq\kappa}2\rightarrow {}^{\leq\kappa}2$ be the function defined by $f_\alpha(x)(i)=x(i)$ for all $i\neq \alpha$ and $f_\alpha(x)(\alpha) \neq x(\alpha)$. 
A subtree $T$ of ${}^{<\kappa}\kappa$ is called \emph{pruned} if for every $t\in T$, there exists some $x\in [T]$ such that $t\subseteq x$. 

\begin{lemma}\label{lemma Baire trees and G0 counterexamples} 
Suppose that $T$ is a pruned subtree of $2^{<\kappa}$ such that 
\begin{enumerate-(a)} 
\item\label{lemma Baire trees and G0 counterexamples a} 
$[T]$ is $\lambda$-Baire for all cardinals $\lambda<\kappa$. 

\item\label{lemma Baire trees and G0 counterexamples b}
There are cofinally many $t\in T$ such that $t\in S\cap T$ and for $\alpha=\len(t)$, $f_\alpha(T_t)\subseteq T_t$. 
\end{enumerate-(a)} 
Then there exists no continuous homomorphism from $G_0^\kappa{\upharpoonright} [T]$ to $G_{\mathrm{wo}}$. 
\end{lemma} 
\begin{proof} 
Suppose that $f\colon [T]\rightarrow \tk$ is a continuous homomorphism from $G_0^\kappa{\upharpoonright}[T]$ to $G_{\mathrm{wo}}$. 
We consider $T$ as a forcing with the reverse order. 

\begin{claim*} 
$T$ is ${<}\kappa$-distributive. 
\end{claim*} 
\begin{proof} 
Suppose that $\lambda<\kappa$ is a cardinal and $(D_i\mid i<\lambda)$ is a sequence of open (upwards closed) dense subsets of $T$. 
Then $U_i=\{x\in [T] \mid x \in N_t \text{ for some $t\in D_i$} \}$ is an open dense subset of $[T]$ for each $i<\lambda$. 
Since $[T]$ is $\lambda$-Baire by \ref{lemma Baire trees and G0 counterexamples a}, $\bigcap_{i<\lambda} U_i$ is dense in $[T]$. 
Suppose that $t\in T$. 
Pick some $x\in \bigcap_{i<\lambda} U_i$ with $t\subseteq x$. 
Then there exist $t_i \in D_i$ with $t_i\subseteq x$ for all $i<\lambda$. 
Thus $t\subseteq \bigcup_{i<\lambda} t_i \in \bigcap_{i<\lambda} D_i$ by the definition of $U_i$, using that $\kappa$ is regular. 
\end{proof} 

Let $x\in \tk$ be $T$-generic over $V$. 
Let $x_0\in \tk$ code the set of pairs $(s,t)\in {}^{<\kappa}2\times {}^{<\kappa}2$ such that $f(N_s)\subseteq N_t$. 
Let $\bar{f}$ denote the partial continuous function on $\tk$ in $V[x]$ coded by $x_0$. 
Let $Gen_T\subseteq \tk$ denote the set of $T$-generics over $V$ in $V[x]$. 

\begin{claim*} 
$Gen_T\subseteq \dom(\bar{f})$. 
\end{claim*} 
\begin{proof} 
Suppose that $y\in Gen_T$ and $\alpha<\kappa$. It suffices to show that there exists some $(s,t)\in x_0$ such that $s\subseteq y$ and $t\in {}^\alpha 2$. 
To see this, let $D$ denote the set of $u\in T$ such that there exists some $s\subseteq u$ and some $t\in {}^\alpha2$ with $(s,t)\in x_0$. 
Since $f$ is continuous, $D$ is dense. 
Since $y$ is generic, it has an initial segment in $D$ and the claim follows. 
\end{proof}

Note that if $T$ is ${<}\kappa$-closed, then $\bar{f}$ is a total function in $V[x]$, but  this is not the case in general. 

\begin{claim*} 
$\bar{f}$ is a partial homomorphism from $G_0^\kappa{\upharpoonright}Gen_T$ to $G_{\mathrm{wo}}$. 
\end{claim*} 
\begin{proof} 
Suppose that $y\in Gen_T$ and $z$ is such that $(y,z)\in G_0^\kappa$. 
Then $z\in Gen_T$ and there is some $r\in S$ such that $r^\smallfrown i\subseteq y$ and $r^\smallfrown j\subseteq y$ for $i\neq j$. 
We can assume that $i=0$. 
Let $\alpha=\len(r)$. 
We want to show that $(f(y),f(z))\in G_{\mathrm{wo}}$. 
To see this, we first argue that $f(y)\neq f(z)$. 
Let $D$ be the set of all $t\in T$ such that $r^\smallfrown 0\subseteq t$ and there are incompatible $u,v\in {}^{<\kappa}2$ such that $(t,u),(f_\alpha(t),v)\in x_0$. 
Since $D$ is dense below $r^\smallfrown 0$, $y$ has an initial segment in $D$.
Since $z=f_\alpha(y)$, we have $f(y)\neq f(z)$ as required. 
Since $G_{\mathrm{wo}}$ is a closed subset of $({}^\kappa2 \times {}^\kappa2)\setminus \Delta$ and $\bar{f}$ is a continuous extension of $f$, it now suffices to find $(a,b)\in \dom(f)$ arbitrarily close to $(y,z)$ with $(f(a),f(b))\in G_{\mathrm{wo}}$. 
To this end, suppose that $\gamma<\kappa$. 
Let $D_\gamma$ be the set of all $u\in T$ such that 
\todo[color=yellow!20]{actually some of the notation can be simplified everywhere if we assume $x_0$ collects \emph{all} pairs $(s,t)$ such that $f(N_s)\subseteq N_t$. we can replace $s$ by $u$ etc.} 
there exist $s\subseteq u$, $t\subseteq f_\alpha(u)$ and $v,w\in {}^\gamma 2$ with $(s,v), (t, w) \in x_0$. 
Since $f$ is continuous, $D_\gamma$ is dense above $r^\smallfrown 0$ and hence $y{\upharpoonright}\beta\in D_\gamma$ for some $\beta<\kappa$. 
Let $a\in \tk$ be a ground model element that extends $y{\upharpoonright} \beta$ and $b=f_\alpha(a)$. 
For $s, t, u, v$ as in the definition of $D_\gamma$, we have $s\subseteq a{\upharpoonright}\beta$ and $t\subseteq b{\upharpoonright}\beta$. 
Then $(a,b)\in G_0^\kappa$ and hence $(f(a),f(b))\in G_{\mathrm{wo}}$. 
Then $f(a){\upharpoonright}\gamma=f(y){\upharpoonright}\gamma=v$ and $f(b){\upharpoonright}\gamma=f(z){\upharpoonright}\gamma=w$ as required. 
\end{proof} 

Note that if $T$ is ${<}\kappa$-closed, then the previous claim holds by $\Pi^1_1({}^\kappa2)$-absoluteness. 

Now work in $V[x]$. 
The next argument is similar to the proof of Lemma \ref{no selector for E0kappa}. 
Suppose $p\subseteq x$ forces that $\bar{f}(x)$ has order type $\gamma$. 
Note that $x$ has unboundedly many initial segments in $S$ that satisfy \ref{lemma Baire trees and G0 counterexamples b} by genericity. 
Hence there is some $\alpha<\kappa$ such that $p\subseteq x{\upharpoonright}\alpha \in S$. 
Let $y=f_\alpha(x)$. 
Since $f_\alpha(T_t)=T_t$ by \ref{lemma Baire trees and G0 counterexamples b}, $y$ is $T$-generic over $V$ and $V[x]= V[y]$. 
Since $p\subseteq y$, $\bar{f}(y)$ has order type $\gamma$. 
Since $(x,y)\in G_0^\kappa$ and $f$ is a homomorphism, $(f(x),f(y))\in G$. 
Thus $f(x)$ is a strict end extension of $f(y)$ or conversely. 
But this contradicts the fact that both have order type $\gamma$. 
\end{proof} 

Proposition \ref{prop graph of wellorders} and Lemma \ref{lemma Baire trees and G0 counterexamples} for $T={}^{<\kappa}2$ show that $G_{\mathrm{wo}}$ is a counterexample to the $G_0^\kappa$-dichotomy. 

\begin{theorem}
\label{G0 dichotomy fails for Gwo} 
The $G_0$-dichotomy fails for $G_{\mathrm{wo}}$. 
\end{theorem}

\begin{remark} 
The analogous graph $G_{\mathrm{wo}}$ for $\kappa=\omega$ is a $\Pi^1_1$ counterexample to the $G_0$-dichotomy. 
To see this, one can use a simplified form of the proof of Lemma \ref{lemma Baire trees and G0 counterexamples}. 
The first two claims hold by $\Pi^1_1$-absoluteness and the proof of the last claim is virtually unchanged. 
\end{remark}

\section{An induced subgraph in $L$} 
\label{section induced subgraph}

One might ask 
\todo[color=yellow!30]{can the following be done for $L[U]$ and other fine structural inner models?}
if there are simpler counterexamples than the graph $G_{\mathrm{wo}}$, for instance induced subgraphs of $G_0^\kappa$ on closed sets. 
The below arguments use that $V=L$, where $L$ denotes the constructible universe $L$. 
A \emph{slim} subtree of ${}^{<\kappa}2$ is one such that $|\mathrm{Lev}_\alpha(T)|\leq\alpha$ for stationarily many $\alpha<\kappa$. 
The next lemma follows from the proof of \cite[Proposition 7.2]{lucke2016hurewicz}. 
The latter is based on an idea from \cite[Section 3]{friedman2015failures}.

\begin{lemma}\label{lemma Kurepa tree in L} 
Suppose that $V=L$ and $\kappa$ is an uncountable regular cardinal. 
There exists a pruned slim $\kappa$-Kurepa subtree $T$ of ${}^{<\kappa}2$ such that 
\begin{enumerate-(1)} 
\item\label{lemma Kurepa tree in L 1}
$[T]$ is $\kappa$-Baire. 

\item\label{lemma Kurepa tree in L 2}
For cofinally many $t\in T$ and $\alpha=\len(t)$, $f_\alpha(T)\subseteq T$. 
\end{enumerate-(1)} 
\end{lemma} 
\begin{proof}[Proof sketch] 
If $\kappa=\mu^+$ is a successor cardinal, let $\nu=\cof(\mu)$. 
Otherwise let $\nu=\omega$. 
Let $\cof_\nu^\kappa$ denote the set of ordinals $\alpha<\kappa$ with $\cof(\alpha)=\nu$. 
Let 
$$F(\alpha)= \min\{ \gamma \mid \alpha < \gamma <\kappa,\ L_\gamma\models \ZFC^-,\  \cof(\alpha)^{L_\gamma}=\nu \}$$ 
for $\alpha<\kappa$ with $\cof(\alpha)=\nu$ and 
$$ T = \{ t \in {}^{<\kappa}2 \mid \forall \alpha \in  \cof_\nu^\kappa\cap \Lim \ t{\upharpoonright}\alpha \in L_{F(\alpha)} \}. $$

$T$ is clearly a slim pruned $\kappa$-tree that satisfies \ref{lemma Kurepa tree in L 2}. 
The proof of \ref{lemma Kurepa tree in L 1} for $T$ is identical to that in \cite[Claim 7.2.2]{lucke2016hurewicz}. 
In particular, $T$ has at least $\kappa^+$ many branches, so it is a $\kappa$-Kurepa tree. 
\end{proof}

The next lemma shows that the range of any reduction of $G_0^\kappa$ contains a $\kappa$-perfect subset. 

\begin{lemma}\label{injective hom G0} 
Suppose that $G$ is a graph on $\tk$ and there exists a continuous homomorphism $f\colon \tk\rightarrow \tk$ from $G_0^\kappa$ to $G$. 
Then there exists a $\kappa$-perfect tree $T$ such that $f{\upharpoonright}[T]$ is injective. 
\end{lemma} 
\begin{proof} 
We construct order preserving families $(t_s\mid s\in {}^{<\kappa}2)$ and $(u_s \mid s\in {}^{<\kappa}2)$ such that $f(N_{t_s})\subseteq N_{u_s}$ for all $s\in {}^{<\kappa}2$ and $(t_s\mid s\in {}^{<\kappa}2)$ is incompatibility preserving. 
Let $t_\emptyset=u_\emptyset=\emptyset$. 
In limit steps, we form unions. 
Hence the required properties are preserved. 
In successor steps, suppose that $\alpha<\kappa$ and $t_s$, $u_s$ have been constructed for all $s\in {}^\alpha 2$. 
Suppose that $s\in {}^\alpha 2$. 
Pick some $r\supseteq s$ in $S$ by density of $S$. 
Since $f$ is a reduction of $G_0^\kappa$ and $(r^\smallfrown0{}^\smallfrown 0^\kappa,r^\smallfrown 1{}^\smallfrown 0^\kappa)$ is an edge, there exist $\beta<\kappa$ and incompatible $u_{s^\smallfrown 0}$, $u_{s^\smallfrown 1}$ extending $u_s$ such that $f(N_{t_{r^\smallfrown i{}^\smallfrown 0^\beta}})\subseteq N_{u_{s^\smallfrown i}}$ for $i=0,1$. 
Let $t_{s^\smallfrown i}=r^\smallfrown i{}^\smallfrown 0^\beta$ for $i=0,1$. 
Finally, define $f\colon \tk\rightarrow \tk$ by letting $f(x)=\bigcup_{\alpha<\kappa} t_{x{\upharpoonright}\alpha}$. Then $f$ is as required. 
\end{proof} 

Note that for any $t\in {}^{<\kappa}2$, there is an edge of $G_0^\kappa$ contained in $N_t$. 
Therefore the range of a continuous homomorphism as in Lemma \ref{injective hom G0} is contained in the closure of the set of vertices of $G$ with at least one edge. 
It follows that for any $\kappa$-Kurepa subtree $T$ of ${}^{<\kappa}2$, any graph on $[T]$ that is not $\kappa$-Borel $\kappa$-colorable is a counterexample to the $G_0^\kappa$-dichotomy. 
For example, this holds for the complete graph on $[T]$. 

\begin{theorem}
\label{counterexample in L} 
Suppose that $V=L$ and $\kappa$ is an uncountable regular cardinal. 
Then there exists a pruned slim $\kappa$-Kurepa subtree of ${}^{<\kappa}2$ such that 
\begin{enumerate-(1)} 
\item\label{counterexample in L 1} 
$G_0^\kappa{\upharpoonright}[T]$ is a counterexample to the $G_0^\kappa$-dichotomy. 

\item\label{counterexample in L 2} 
There exists no continuous homomorphism from $G_0^\kappa{\upharpoonright}[T]$ to $G_{\mathrm{wo}}$. 

\end{enumerate-(1)} 
\end{theorem} 

\begin{proof} 
Let $T$ be a tree as in Lemma \ref{lemma Kurepa tree in L}. 

\ref{counterexample in L 1} 
$G_0^\kappa{\upharpoonright}[T]$ does not admit a $\kappa$-Borel $\kappa$-coloring by Lemma \ref{comeager not discrete}. 
So suppose towards a contradiction that $f\colon \tk\rightarrow \tk$ is a continuous homomorphism from $G_0^\kappa$ to $G_0^\kappa{\upharpoonright}[T]$. 
Note that $\ran(f)\subseteq [T]$, since every pair $(x,x)$ with $x\in \tk$ is a limit of edges. 
There exists a $\kappa$-perfect subtree $U$ of ${}^{<\kappa}2$ such that $f{\upharpoonright}[U]$ is injective by Lemma \ref{injective hom G0}. 
Then $T$ contains a $\kappa$-perfect subtree by \cite[Lemma 2.9]{lucke2016hurewicz}. 
But this contradicts the fact that $T$ is a slim $\kappa$-Kurepa tree. 

\ref{counterexample in L 2} 
By Lemmas \ref{lemma Baire trees and G0 counterexamples} and \ref{lemma Kurepa tree in L}. 
\end{proof} 

Note that \ref{counterexample in L 2} shows that the statement of the $G_0^\kappa$-dichotomy remains false if $G_0^\kappa$ is replaced by $G$.


\section{A connection with Silver's dichotomy} 
\label{section a connection with Silver's dichotomy}

We will provide further counterexamples to the $G_0^\kappa$-dichotomy. 
Silver's dichotomy for an equivalence relation on $\tk$ states that either has at most $\kappa$ equivalence classes, or otherwise there exists a $\kappa$-perfect set of pairwise inequivalent elements of $\tk$. 
Silver's dichotomy is consistent for all $\kappa$-Borel equivalence relations on $\tk$ relative to $0^\sharp$ by a result of Friedman \cite[Theorem 8]{friedman2014consistency}. 
By Theorem \ref{theorem G0 dichotomy fails} the $G_0^\kappa$-dichotomy may fail at a much lower complexity than Silver's dichotomy. 
We now discuss how any counterexample to Silver's dichotomy gives rise to one for the $G_0^\kappa$-dichotomy. 
The counterexamples are stronger in the sense that the condition that the coloring is $\kappa$-Borel can be dropped. 
A subset $A$ of $(\tk)^n$ is called \emph{strongly $\kappa$-bianalytic} if it is both $\kappa$-analytic and $\kappa$-coanalytic and the definitions remain equivalent in $\Add(\kappa,1)$-generic extensions. 
For example, any $\kappa$-Borel set is strongly $\kappa$-bianalytic, since the canonical $\kappa$-analytic definition of a $\kappa$-Borel set (from the proof that $\kappa$-Borel sets are $\kappa$-analytic) is the same in all generic extensions. 

\begin{lemma}\label{bianalytic baire}
Any 
$\Delta^1_1$ subset of ${}^\kappa 2$ with parameters in the ground model has the $\kappa$-Baire property in any $\Add(\kappa,1)$-generic extension. 
\end{lemma} 
\begin{proof} The proof is similar to \cite[Theorem 49.7]{MR3235820}. 
Suppose that $X$ denotes the $\Delta^1_1$ set in the $\Add(\kappa,1)$-generic extension $V[G]$. 
Suppose that $s\in \Add(\kappa,1)$. 
It suffices to find $t\supseteq s$ such that $\dot{X}$ is dense in $N_t$ or ${}^\kappa 2$ is dense in $N_t$. 
To see this, pick an $\Add(\kappa,1)$-generic $x\supseteq s$ such that $V[G]=V[x]$ and let $\dot{X}$ be an $\Add(\kappa,1)$-name for $X$ relative to $x$. 
Let $\dot{x}$ be a name for the $\Add(\kappa,1)$-generic. 
Then there exists some $t$ with $s\subseteq t \subseteq x$ and $t\Vdash \dot{x}\in \dot{X}$ or $t\Vdash \dot{x}\notin \dot{X}$. 
First suppose that $t\Vdash \dot{x}\in \dot{X}$ 
and let $\dot{T}$ be a subtree of ${}^{<\kappa}2\times {}^{<\kappa}2$ such that $1_{\Add(\kappa,1)}$ forces $p[T]=\dot{X}$. 
Then there is a name $\sigma$ such that $t\Vdash (\dot{x},\sigma)\in [T]$. 
For each $i<\kappa$, let $D_i$ be the set of conditions $p\supseteq t$ deciding $\sigma{\upharpoonright}i$. 
Then $D_i$ is open and dense below $t$ for all $i<\kappa$, so  $\bigcap_{i<\kappa} D_i$ is $\kappa$-comeager in $N_t$ and contained in $X$. 
The case $t\Vdash \dot{x}\notin \dot{X}$ is analogous. 
\end{proof}

\begin{definition}
    A $\kappa$-Cantor scheme on a set $X$ is a family $(A_s)_{s \in {}^{<\kappa} 2}$ such that for all $s,t \in {}^{<\kappa}2$: 

    \begin{enumerate-(1)}
        \item 
        $A_{s^\frown 0}  \cap A_{s^\frown 1} = \emptyset$.

        \item 
         If $s \subseteq t$ then $A_t
        \subseteq A_s$. 
    \end{enumerate-(1)}
    We say that the scheme has \emph{vanishing diameter} if $\vert  \bigcap_{\alpha < \kappa} A_{x \restriction \alpha}  \vert =1$ for all $x \in \tk$. 
\end{definition}

We use the following weaker form of $\diamondsuit_\kappa$ from \cite[Definition 5.4.1]{schlicht2023open}. 

\begin{definition}
$\lozenge^i_\kappa$ asserts the existence of a sequence $\langle \mathcal{A}_\alpha \mid \alpha < \kappa \rangle$ of families $\mathcal{A}_\alpha \subseteq {}^\alpha \kappa$ such that

\begin{enumerate-(a)}
    \item $\vert \mathcal{A}_\alpha \vert < \kappa$ for all $\alpha < \kappa$, and 

    \item for all $y \in {}^\kappa \kappa$, the set $\{ \alpha < \kappa \mid y \restriction \alpha \in \mathcal{A}_\alpha \}$ is stationary in $\kappa$.
\end{enumerate-(a)}
\end{definition} 

Note that $\lozenge^i_\kappa$ holds in particular at every inaccessible $\kappa$. 
Friedman, Miller, Motto Ros and T\"ornquist extended Mycielski's theorem to ${}^\kappa2$ assuming $\diamondsuit_\kappa$ \cite{millersuslinspaces}. 
We slightly strengthen their result by replacing $\diamondsuit_\kappa$ with $\diamondsuit^i_\kappa$. 
For $\lambda<\kappa$, let $\Delta_\lambda$ denote the set of all non-injective sequences $x \in {}^\lambda({}^\kappa 2)$.  

\begin{proposition}
\label{mycielski}
Assume $\lozenge^i_\kappa$. Then for any $\lambda < \kappa$ and any comeager $R \subseteq {}^{{}^\lambda} (\tk) $, there exists a $\kappa$-perfect subset $C \subseteq \tk$ such that ${}^\lambda C \setminus \Delta_\lambda \subseteq R$. 
\end{proposition}
\begin{proof}
Let $( \mathcal{A}_\alpha \mid \alpha < \kappa )$  be a $\lozenge^i_\kappa$-sequence, and let $\langle D_\alpha \mid \alpha < \lambda \rangle$ be a decreasing sequence of open dense sets with $D_0 = \tk$ and such that $\bigcap_{\alpha <\lambda} D_\alpha \subseteq R$. 

By recursion on the levels of ${}^{<\kappa}2$, we construct a Cantor scheme $(U_s)_{s \in  {}^{<\kappa} 2}$ with vanishing diameter made of basic open sets. In particular, we let $U_s = N_{h(s)}$ for an order and incompatibility preserving function $h: {}^{< \kappa} 2 \to {}^{<\kappa} 2 $. 
If $\alpha = \lambda \cdot \alpha$ for some $\alpha < \kappa$, then we read off any $s \in \mathcal{A}_\alpha$ as $(s_\gamma)_{\gamma < \nu} \in {}^{{}^\nu} ({}^\alpha 2)$ for some $\nu<\kappa$. 
We enumerate $\mathcal{A}_\alpha$ as $(s^{\alpha,\beta})_{\beta < \nu_\alpha}$ and write $s^{\alpha,\beta}=(s_\gamma^{\alpha,\beta})_{\gamma < \lambda} \in {}^{{}^\lambda} ({}^\alpha 2)$. 
For such $\alpha$'s, we can moreover ensure that for every  $s \in \mathcal{A}_\alpha$ that reads as an injective sequence $(s_\gamma)_{\gamma < \lambda}$, we have $\prod_{\gamma < \lambda} U_{s_\gamma} \subseteq D_\alpha$.

At successors $\alpha=\delta+1$, pick arbitrary incompatible extensions $h(s^\smallfrown 0)$ and $h(s^\smallfrown 1)$ of $h(s)$ for each $s\in {}^\delta 2$, and form unions at limits $\alpha$ that do not satisfy $\alpha=\lambda\cdot\alpha$.  

Suppose that $\alpha=\lambda \cdot \alpha$ and the above properties are preserved for all levels $< \alpha$. 
Let $h^*_0(s)=\bigcup_{r\subsetneq s} h(r)$ for all $s\in {}^\alpha 2$. 
We will iteratively define increasing sequences $(h^*_\beta(s)\mid \beta<\nu_\alpha)$ for all $s\in {}^\alpha 2$ such that for each $\beta<\nu_\alpha$, we have $\prod_{\gamma < \lambda} N_{h^*(s^{\alpha,\beta}_\gamma)} \subseteq D_\alpha$. 
In limit steps, we form unions. 
In successor steps, suppose that $h^*_\beta(s)$ has been defined for all $s\in {}^\alpha 2$. 
Let $h^*_{\beta+1}(s)=h^*_\beta(s)$ for all $s\in {}^\alpha 2$ such that $s\neq s^{\alpha,\beta}_\gamma$ for all $\gamma<\lambda$. 
Using the density of $D_\alpha$, find $h^*_{\beta+1}(s^{\alpha,\beta}_\gamma)\leq h^*_\beta(s^{\alpha,\beta}_\gamma)$ for all $\gamma<\lambda$ such that $\prod_{\gamma < \lambda} N_{h^*_{\beta+1}(s^{\alpha,\beta}_\gamma)} \subseteq D_\alpha$. 
Then $h(s)=\bigcup_{\beta<\nu_\alpha} h^*_\beta(s)$ for each $s\in {}^\alpha 2$ is as required. 

Let $f(x)=\bigcup_{\alpha<\kappa} h(x{\upharpoonright} \alpha)$ for $x\in {}^\kappa 2$. 
Then $C=\ran(f)$ is $\kappa$-perfect. 
The application of the $\diamondsuit^i_\kappa$-sequence ensures that any pair $(x,y)\in C$ with $x\neq y$ is considered unboundedly often. 
Since $\langle D_\alpha \mid \alpha<\lambda\rangle$ is decreasing, this ensures ${}^\lambda C \setminus \Delta\subseteq R$. 
\end{proof}

To any binary relation $E$ on ${}^\kappa2$, we associate the graph $G_E=({}^\kappa2\times {}^\kappa2) \setminus E$.

\begin{lemma}\label{theorem bianalytic Silver G0} 
Suppose that $E$ is a strongly $\kappa$-bianalytic equivalence relation on $\tk$ that contradicts Silver's dichotomy in some $\Add(\kappa,1)$-generic extension. 
Then $G_E$ is a counterexample to the $G_0^\kappa$-dichotomy. 
\end{lemma} 
\begin{proof} 
Write $G=G_E$. 
Since $E$ has at least $\kappa^+$ many equivalence classes, $G$ is not $\kappa$-colorable. 
Towards a contradiction, suppose that the $G_0^\kappa$-dichotomy holds for $G$. 
Then there exists a continuous homomorphism $f\colon \tk\rightarrow \tk$  from  $G_0^\kappa$ to $G$. 
The canonical extension of $f$ remains a homomorphism from $G_0^\kappa$ to $G$ in any $\Add(\kappa,1)$-generic extensions by $\Pi^1_1(\tk)$-absoluteness \cite[Proposition 7.3]{lucke2012definability}. 
Now suppose that $x\in \tk$ is $\Add(\kappa,1)$-generic over $V$. 
From now on work in $V[x]$ and write $f, E, G$ for the canonical extensions. 
The next claim suffices, since $E$ was assumed to contradict Silver's dichotomy in $V[x]$. 

\begin{claim*} 
There exists a $\kappa$-perfect $G$-clique. 
\end{claim*} 
\begin{proof} 
The inverse image $G^f=(f\times f)^{-1}(G)$ is an equivalence relation with the $\kappa$-Baire property by Lemma \ref{bianalytic baire}. 
Note that the Kuratowski-Ulam theorem generalizes to ${}^\kappa2$ assuming $\kappa^{<\kappa}=\kappa$ by virtually the same proof 
\cite[Theorem 8.41]{kechris2016descriptive}. 
We first claim that $G^f$ is $\kappa$-comeager. 
Otherwise, by the analogue to the Kuratowski-Ulam theorem for $2^\kappa$, 
there exists some $a\in 2^\kappa$ such that the slice $G^f_a$ has the $\kappa$-Baire property but is not $\kappa$-comeager. 
Since $2^\kappa\setminus G^f_a$ is not $G_0^\kappa$-independent by Lemma \ref{comeager not discrete}, there exist $b,c\in \tk\setminus G^f_a$ with $(b,c)\in G_0^\kappa$. 
Since $\tk\setminus G^f_a$ is an equivalence class of $(\tk\times \tk)\setminus G_f$, 
$(b,c) \notin G^f$ and $(f(b),f(c))\notin G$. 
But this contradicts the fact that $(b,c)\in G_0^\kappa$ and $f$ is a homomorphism from $G_0^\kappa$ to $G$. 
Since $G^f$ is $\kappa$-comeager and $\diamondsuit_\kappa$ holds, 
Proposition \ref{mycielski} applied to $G^f$ produces a $\kappa$-perfect set $C$ of pairwise $G^f$-connected elements of $\tk$. 
For all distinct $x, y$ in $C$, we have $(f(x),f(y))\in G$. 
In particular, $f$ is injective and hence $f[C]$ contains a $\kappa$-perfect set by the argument in  \cite[Lemma 2.9]{lucke2016hurewicz}. 
\end{proof} 
This completes the proof of the lemma. 
\end{proof}

The next two propositions yield further counterexamples to the $G_0^\kappa$-dichotomy via Theorem \ref{theorem bianalytic Silver G0}.

\begin{proposition}
\label{prewellorder counterexample} 
There exists a strongly $\kappa$-bianalytic pre-wellorder on $\tk$ of length $\kappa^+$ that induces a counterexample $E$ to Silver's dichotomy in any $\Add(\kappa,1)$-generic extension. 
In particular, $G_E$ is a counterexample to the $G_0^\kappa$-dichotomy. 
\end{proposition}
\begin{proof} 
We pre-wellorder the elements of $\tk$ coding wellorders via a pairing function on $\kappa$ by comparing their order types. 
Since wellfoundedness and order type are absolute, the $\kappa$-analytic and $\kappa$-coanalytic definitions agree in any outer model. 
Moreover, in any $\Add(\kappa,1)$-generic extension, there is no perfect set of elements of $\tk$ of pairwise different order types. 
This would induce a wellorder of $\tk$ with the $\kappa$-Baire property by Lemma \ref{bianalytic baire}, but such wellorders cannot exist by an argument analogous to the one for $\omega$ in \cite[Theorem 8.48]{kechris2016descriptive}. 
In detail, one works with the least non-$\kappa$-meager initial segment and notes that the slices are all $\kappa$-meager, thus arriving at a contradiction via the analogue to the Kuratowski-Ulam theorem for ${}^\kappa 2$. 
\end{proof}

The next proposition yields a different kind of counterexample in $L$ derived from a long $\kappa$-Borel wellorder. 
By the induced equivalence relation, we mean equality on the set and a single equivalence class for the rest. 

\begin{proposition}
\label{counterexample wellorder in L} 
Suppose that $V=L$ and $\kappa$ is an uncountable regular cardinal. 
There exists a $\kappa$-Borel wellorder of length $\kappa^+$ on a $\kappa$-Borel set that does not acquire new elements in outer models. 
If $E$ denotes the induced equivalence relation, then $G_E$ is a counterexample to the $G_0^\kappa$-dichotomy. 
\end{proposition} 
\begin{proof} 
The following direct argument could be alternatively formulated using Jensen's analysis of $\Sigma_1$-definability and the fine structure of $L$. 
For unboundedly $\kappa\leq\alpha<\kappa^+$, a new subset of $\kappa$ appears in $L_{\alpha+1}\setminus L_\alpha$, and any such $\alpha$ has size $\kappa$ in $L_{\alpha+1}$ by acceptability of the $L$-hierarchy \cite[Theorem 1]{boolos1969degrees}. 
Then for some $n\in\omega$ and unboundedly many $\kappa\leq \alpha<\kappa^+$, there exists a 
$\Sigma_n^{L_\alpha}$-definable surjection from $\kappa$ onto $L_\alpha$ and thus a $\Sigma_n^{L_\alpha}$-definable relation on $\kappa\times \kappa$ coding $L_\alpha$ such that the subset $\kappa$ of $L_\alpha$ is coded into the even ordinals via the canonical pairing function on $\kappa$. 
Pick the $L$-least code $x_\alpha$ for $L_\alpha$ for each such $\alpha$ to obtain a $\Delta_{n+1}$-definable set $C\subseteq \tk$. 
Since $x_\beta$ codes $L_\beta$ and $x_\alpha\in L_\beta$ for $\alpha<\beta$, the set of pairs $(x_\alpha, x_\beta)$ with $\kappa\leq \alpha<\beta<\kappa^+$ is a $\Delta_{n+1}$ wellorder of ${}^\kappa 2$ with order type $\kappa^+$. 
Since the definition refers to $L$-levels, it is absolute to all outer models. 

Recall that there exists no $\kappa$-Borel wellorder on ${}^\kappa2$, as we have sketched in the proof of Proposition \ref{prewellorder counterexample}. 
Using Theorem \ref{theorem bianalytic Silver G0}, it therefore follows from the properties of $E$ that $G_E$ is a counterexample to the $G_0^\kappa$-dichotomy. 
\end{proof}

Conversely, the graphs $G_E$ in Propositions \ref{prewellorder counterexample} and \ref{counterexample wellorder in L} do not admit (even discontinuous) homomorphisms to $G_0^\kappa$ nor to $G_{\mathrm{wo}}$, since both of these graphs do not contain cliques of size $\kappa^+$. 
Hence the variant of the $G_0^\kappa$-dichotomy with $G_0^\kappa$ replaced by $G_E$ fails.




\todo[color=yellow!30]{Can we show that $G_0^\kappa{\upharpoonright}[T]$, with $T$ the $\kappa$-Kurepa tree above, does not cont. homomorph to this graph? Maybe a variant of the above argument works. But above we used that $G_{\mathrm{wo}}$ is closed.\\
}

\todo[color=yellow!30]{Is the previous counterexample incomparable with the graph $G_{\mathrm{wo}}$? 
$G$ cannot cont. homomorph to $G_{\mathrm{wo}}$, since otherwise it would contain a copy of a complete graph of size $\kappa^+$. \\ 
\ \\ 
Can $G_{\mathrm{wo}}$ homomorph to $G_E$ (for both $G_E$)? 
} 

\todo[color=yellow!30]{I think we can show that any subgraph of $G$, whose ``in the same connected component'' relation is $\kappa$-Borel, either does not homomorph to $G_0^\kappa$, or it is $\kappa$-Borel $\kappa$-colorable. 
The idea is that every $x$ can see the $L$-previous $y$ in its connected component and what they are mapped to. On each component the homomorphism is a coloring and we need to stitch them together in a $\kappa$-Borel way. 
\\ 
\ \\ 
It's not clear  how to do this for graphs that homomorph to $G$. If we can do it, we have shown there is no analogue for the $G_0$-dichotomy for any graph. 
}

\section{Open questions} 

The above results suggest to ask whether a different graph can play the role of $G_0^\kappa$. 
The graphs $G_{\mathrm{wo}}$, $G{\upharpoonright}[T]$ and $G_E$  cannot do this by Proposition \ref{prop graph of wellorders}, Theorem \ref{counterexample in L} and the discussion after Proposition \ref{counterexample wellorder in L}. 

\begin{question}
\label{question  replace G0} 
Is it consistent with $\kappa^{<\kappa}=\kappa$ that there a closed graph $G$ on ${}^\kappa 2$ that is least among non-$\kappa$-colorable closed graphs with respect to continuous homomorphisms? 
\end{question}

We further ask whether the $G_0^\kappa$-dichotomy for induced subgraphs of $G_0^\kappa$ holds after the Levy collapse of an inaccessible cardinals above $\kappa$ to $\kappa^+$. 
This is a test question towards a strategy for a negative answer to Question \ref{question  replace G0} in a $\kappa$-Levy model.


\bibliographystyle{alpha}
\bibliography{references}

\newcommand{\etalchar}[1]{$^{#1}$}
\begin{thebibliography}{FMM{\etalchar{+}}12}

\bibitem[BP69]{boolos1969degrees}
George Boolos and Hilary Putnam.
\newblock Degrees of unsolvability of constructible sets of integers.
\newblock {\em The Journal of Symbolic Logic}, 33(4):497--513, 1969.

\bibitem[FHK14]{MR3235820}
Sy-David Friedman, Tapani Hyttinen, and Vadim Kulikov.
\newblock Generalized descriptive set theory and classification theory.
\newblock {\em Memoirs of the American Mathematical Society}, 230(1081):vi+80,
  2014.

\bibitem[FK15]{friedman2015failures}
Sy-David Friedman and Vadim Kulikov.
\newblock Failures of the {S}ilver dichotomy in the generalized {B}aire space.
\newblock {\em The Journal of Symbolic Logic}, 80(2):661--670, 2015.

\bibitem[FMM{\etalchar{+}}12]{millersuslinspaces}
Sy~Friedman, Luca {Motto Ros}, Ben Miller, Philipp Schlicht, and Asger
  Törnquist.
\newblock Descriptive set theory in generalized {S}ouslin spaces.
\newblock Unpublished manuscript, 2012.

\bibitem[Fri14]{friedman2014consistency}
Sy-David Friedman.
\newblock Consistency of the {S}ilver dichotomy in generalised {B}aire space.
\newblock {\em Fundamenta Mathematicae}, 227:179--186, 2014.

\bibitem[Gao08]{gao2008invariant}
Su~Gao.
\newblock {\em Invariant descriptive set theory}.
\newblock CRC Press, 2008.

\bibitem[Kec95]{kechrisclassical}
Alexander~S. Kechris.
\newblock {\em Classical descriptive set theory}, volume 156 of {\em Graduate
  Texts in Mathematics}.
\newblock Springer-Verlag, New York, 1995.

\bibitem[KM20]{kechris2016descriptive}
Alexander~S. Kechris and Andrew~S. Marks.
\newblock Descriptive graph combinatorics.
\newblock {\em Book manuscript}, 2020.

\bibitem[KST99]{kechris1999borel}
Alexander~S. Kechris, Slawomir Solecki, and Stevo Todorcevic.
\newblock Borel chromatic numbers.
\newblock {\em Advances in Mathematics}, 141:1--44, 1999.

\bibitem[LMRS16]{lucke2016hurewicz}
Philipp L{\"u}cke, Luca Motto~Ros, and Philipp Schlicht.
\newblock The {H}urewicz dichotomy for generalized {B}aire spaces.
\newblock {\em Israel Journal of Mathematics}, 216(2):973--1022, 2016.

\bibitem[LS20]{lucke2020descriptive}
Philipp L{\"u}cke and Philipp Schlicht.
\newblock Descriptive properties of higher {K}urepa trees.
\newblock {\em To appear in Research Trends in Contemporary Logic (book
  chapter)}, 2020.
\newblock arXiv preprint arXiv:2010.11597.

\bibitem[L{\"u}c12]{lucke2012definability}
Philipp L{\"u}cke.
\newblock {$\Sigma^1_1$}-definability at uncountable regular cardinals.
\newblock {\em The Journal of Symbolic Logic}, 77(3):1011--1046, 2012.

\bibitem[Mil09]{millerparisone}
Benjamin Miller.
\newblock Forceless, ineffective, powerless proofs of descriptive dichotomy
  theorems. {L}ecture {I}: {S}ilver's theorem.
\newblock {\em Lecture notes}, 2009.

\bibitem[Mor23]{moreno2023shelah}
Miguel Moreno.
\newblock Shelah's main gap and the generalized {B}orel-reducibility.
\newblock {\em arXiv preprint arXiv:2308.07510}, 2023.

\bibitem[SS26]{schlicht2023open}
Philipp Schlicht and Dorottya Szir{\'a}ki.
\newblock The open dihypergraph dichotomy for generalized {B}aire spaces and
  its applications.
\newblock {\em Memoirs of the European Mathematical Society}, pages 1--166,
  2026.
\newblock To appear.

\end{thebibliography}

\end{document}